\documentclass[aop]{imsart2}
\usepackage{tikz}
\usepackage{neuralnetwork} 
\RequirePackage{amsthm,amsmath,amsfonts,amssymb}
\RequirePackage[numbers]{natbib}
\RequirePackage[colorlinks,citecolor=blue,urlcolor=blue]{hyperref}
\RequirePackage{graphicx}

\startlocaldefs
\theoremstyle{plain}

\newtheorem{theorem}{Theorem}[section]
\newtheorem{proposition}[theorem]{Proposition}
\newtheorem{lemma}[theorem]{Lemma}
\theoremstyle{remark}
\newtheorem{definition}[theorem]{Definition}
\newtheorem{remark}[theorem]{Remark}

\def \cD {\mathcal{D}}
\def \cE {\mathcal{E}}

\def \cP {\mathcal{P}}

\def \i {\iota}

\def \N {\mathbb{N}}

\def \P {\mathbb{P}}

\def \R {\mathbb{R}}

\endlocaldefs

\begin{document}

\begin{frontmatter}

\title{A non-oriented first passage percolation model and 
  statistical invariance by time reversal}
\runtitle {A non-oriented first passage 
percolation model}

\begin{aug}
\author[A,B]{\fnms{Alejandro F.} \snm{Ram\'irez}\ead[label=e1,mark]{ar23@nyu.edu}},
\author[B]{\fnms{Santiago} \snm{Saglietti}\ead[label=e2,mark]{sasaglietti@mat.uc.cl}}
\and
\author[C]{\fnms{Lingyun} \snm{Shao}\ead[label=e3,mark]{ls5054@nyu.edu}}
\address[A]{NYU-ECNU Institute of Mathematical Sciences,
NYU Shanghai,
\printead{e1}}

\address[B]{Facultad de Matem\'aticas,
Pontificia Universidad Cat\'olica de Chile,
\printead{e2}}

\address[C]{Department of Mathematics,
  University of Toronto
\printead{e3}}
\end{aug}

\begin{abstract}
  We introduce and study a non-oriented first passage percolation model
  having a property of statistical invariance by time reversal. This
  model is defined in a graph having directed edges and the passage times
  associated with each set of outgoing edges from a given vertex are distributed
  according to a generalized Bernoulli--Exponential law and
  i.i.d. among vertices. We derive the statistical invariance property by time
  reversal through a zero-temperature limit of the random walk in Dirichlet
  environment model.
\end{abstract}



\end{frontmatter}

\section{Introduction} We present a non-oriented first 
passage 
percolation (FPP) model  satisfying a statistical invariance property by
time reversal. This property is analogous to the statistical 
invariance 
property satisfied by the random walk in Dirichlet random environment \cite{ST17}.

Random walk in random environment (or RWRE, for short) is a challenging model for which
several fundamental questions remain open, including the proof of a
law of large numbers \cite{Z06, DR14}.  In particular, in the
multidimensional
case, very few explicit computations can be made, so that
for example, there are no general formulas for the velocity,
asymptotic direction,
the diffusion matrix in the case of Gaussian fluctuations, or
the large deviation rate functions. 
The random walk in Dirichlet environment (RWDE) is a particular case of
RWRE in which the marginal law of the
environment
at each site is given by a Dirichlet distribution \cite{ST17}. Sabot
proved in \cite{S11} that RWDE satisfies a property of satistical
invariance by
time reversal. This was used in several works (see \cite{ST17}), to prove
fundamental properties of this model, including explicit conditions
for transient-recurrent behavior in dimension $d=3$ \cite{S11} and
an explicit formula for the asymptotic direction in \cite{T15}.
Also, a particular case of the RWDE model corresponding to
parameters which allow only oriented (or directed) movements of the
random walk, called the Beta random walk, was studied by
Barraquand and Corwin in \cite{BC17}, where they were able to proved Tracy--Widom GUE fluctuations of the quenched
large deviation principle, thus establishing the belonging of this model to the KPZ universality class. Furthermore, 
the zero-temperature limit of the Beta random walk, which is a
first passage percolation model called Exponential first passage
percolation, was also studied in \cite{BC17}, where they showed that the fluctuations
of the first passage times are also of the Tracy--Widom GUE type.

In this article we introduce a zero-temperature version of general
RWDE, which we call the {\it Bernoulli--Exponential first passage
  percolation model}, which turns out to be a non-oriented extension of the exponential
first passage percolation model from \cite{BC17}. The main result of
this article is the proof of a zero-temperature version of the
statistical reversibility under time reversal property of this
model. To our knowledge, this is the first model of FPP type
satisfying such a property.

In principle, the definition of the Bernoulli--Exponential first
passage percolation model can be understood following the principles
of tropical geometry: starting from RWDE, addition is replaced with minimization and
multiplication with addition. However,
due to the presence of infinite sums in the probabilities involved in
the RWDE, this correspondence cannot be made rigorous, and so in
the end here the Bernoulli--Exponential first passage percolation model is
defined and studied directly, without any specific reference to the
RWDE. An exception to this is the proof of the statistical invariance
property under time reversal for the Bernoulli--Exponential first passage
percolation model, which is derived through a limiting procedure from the RWDE.

In what follows, in Section \ref{section-two}, we give a precise
statement of the main results of the article. In Section
\ref{section-three},
we derive the Bernoulli--Exponential first passage exponential model
as a zero-temperature limit of the RWDE. In Section
\ref{section-four},
we present a proof of the statistical invariance by time reversal of
the model. Finally in Section \ref{section-five}, we apply the
statistical
invariance by time reversal property to compute some first passage
times
in specific graphs.

\section{Main results}
\label{section-two}

In order to define this model, we need to introduce first a multinomial version of 
the 
Bernoulli--Exponential random variable.
To this end, let us recall the generalized Bernoulli distribution:
given $M \in \N$ and a probability vector $\mathbf{p}=(p_1,\dots,p_M) \in \R^M$, we say that a random vector $B=(B_1,\dots,B_M)$ has \textit{generalized Bernoulli distribution} of order $M$ with parameters $\mathbf{p}$, and denote it by $B \sim \textrm{Be}_M(\mathbf{p})$, if it satisfies 
\[
\P( B = \mathrm{e}_i )= p_i, \quad i=1,\dots,M, 
\]	where $\mathrm{e}_i$ denotes the $i$-th canonical vector in $\R^M$. We next define the Bernoulli--Exponential distribution of order $M$.

\begin{definition}\label{def:1}
Given $M \in \N$ and a vector $\mathbf{a}=(a_1,\dots,a_M)$ with real positive entries, we define the {\it generalized Bernoulli--Exponential 
  distribution} of order $M$ with parameters $\mathbf{a}$ as the distribution of the random vector 
 \begin{equation}\label{eq:defX}
X=((1-B_1)E_1,\ldots, (1-B_M)E_M), 
\end{equation}
where: 
\begin{enumerate}
\item [$\bullet$] $B=(B_1,\dots,B_M)$ has \textit{generalized Bernoulli distribution} with parameters $M$ and $\mathbf{p}^{(\mathbf{a})}$, where $\mathbf{p}^{(\mathbf{a})}=(p_1^{(\mathbf{a})},\dots,p_M^{(\mathbf{a})})$ is given by 
\begin{equation}\label{eq:defp}
p_i^{(\mathbf{a})}:=\frac{a_i}{\sum_{j=1}^M a_j}; 
\end{equation}
\item [$\bullet$] $E=(E_1,\dots,E_M)$ is a random vector independent of $B$ which has independent entries, with each entry $E_i$ having an Exponential distribution of parameter $a_i$, respectively. 
\end{enumerate} In the sequel, we will write $X \sim \textrm{Be-Exp}_M(\mathbf{a})$ to indicate that $X$ has this distribution. 
\end{definition}

With Definition~\ref{def:1} at our disposal, we now define the {\it generalized Bernoulli--Exponential first 
  passage percolation model} as follows. Let
$G=(V,\cE)$ be a directed graph endowed with a family of positive weights $\mathbf{a}:=(a_e)_{e\in \cE}$.  Assume that $G$ is finite and strongly connected, i.e., that for any $x,y\in 
V$
there exists a path going from $x$ to $y$, where by \textit{path} we understand a finite vector $\pi=(x_0,\ldots,x_n)$ with $x_0,\ldots,  
 x_n\in V$ and $(x_i,x_{i+1})\in \cE$ for all $0\le i\le n-1$.  
For each $x \in V$, define $\cE_x:=\{ e \in \cE :  e=(x,y) \text{ for some }y \in V\}$ to be the set of edges in $\cE$ leaving $x$ and write $M_x:=\# \cE_c$ for their total number.  Now consider a family $\omega=(\omega(x))_{x \in V}$ of independent random vectors where, for each $x \in V$, the vector $\omega(x)=(\omega(x,e))_{e \in \cE_x}$ has generalized Bernolli--Exponential distribution of order $M_x$ with parameters $\mathbf{a}_x:=(a_e)_{e \in \cE_x}$.
We will call the family $\omega$ a {\it 
    generalized 
    Bernoulli--Exponential environment} on the graph $G$ with parameters $\mathbf{a}$, and denote its law by $\P_{\text{BE}}^{(\mathbf{a})}$. Finally, we  
 denote by $\Pi(x,y)$ the set of paths going from $x$ to $y$. 
Now, we define the {\it first passage time between} two points $x,y\in 
G$  by 
\begin{equation}
    T(x,y):=\inf_{\pi\in \Pi(x,y)}\sum_{i=0}^{n-1} w(x_i,\Delta x_i), 
\end{equation} where $\Delta x_i := (x_i,x_{i+1})$. We will call $(T(x,y))_{x,y \in V}$ the generalized Bernoulli-Exponential 
 first passage percolation model on $G$.

In order to state our main result, we first need to introduce the \textit{dual graph} 
of $G=(V,\cE)$, which is defined as 
$\check{G}:=(V,\check{\cE})$, where  
\[
\check \cE=\{(y,x): (x,y)\in \cE\}, 
\]
is the set of reversed edges of $\cE$. In the following, we will write $\check 
e=(y,x)$ for the reversal of the edge $e=(x,y)$. We endow $\check{G}$ with the family of reversed weights $\check{\mathbf{a}}:=(\check{a}_{\check{e}})_{\check{e} \in \check{\cE}}$, where 
\begin{equation}\label{eq:defrev}
\check{a}_{\check e}:=a_e. 
\end{equation}
Furthermore, define the \textit{divergence operator} on the graph $G$ as 
the function $\text{div}: {\mathbb R}^\cE\to{\mathbb R}^V$ given by
\[
 \text{div}(\theta)(x):=\sum_{e \in \cE^x}\theta_e-\sum_{e \in \cE_x}\theta_e, 
\]
where $\cE^x:=\{ e \in \cE : e=(y,x) \text{ for some }y \in V\}$ is the set of edges in $\cE$ arriving at $x$. Finally, we shall say that a 
function 
$f:\cE\to \R$ satisfies the {\it closed loop property} if 
for every path $\pi=(x_0,\ldots, x_n)$ which is closed, i.e., which satisfies $x_n=x_0$, we have that 
\[
\sum_{i=0}^{n-1} f(\Delta x_i)=0. 
\]

Our main result is then the following.
 
\begin{theorem}
  \label{tone}
Let  $G=(V,\cE)$ be a finite strongly connected directed graph endowed with a family of positive weights $\mathbf{a}=(a_e)_{e \in \cE}$ with $\text{div}(\mathbf{a}) = 0$. Then, there exists a coupling $\Upsilon=(\omega, \check{\omega}, \varphi)$ such that:
\begin{enumerate}
\item [i.] $\omega$ is a generalized Bernoulli--Exponential environment on $G$ with parameters $\mathbf{a}$;
\item [ii.] $\check{\omega}$ is a generalized Bernoulli--Exponential environment on $\check{G}$ with parameters $\check{\mathbf{a}}$;
\item [iii.] $\varphi=(\varphi(e))_{e \in \cE}$ is a random function satisfying the closed loop property almost surely;
\item [iv.] $\check{\omega}(y,\check{e})=\varphi(e)+\omega(x,e)$ for all $e=(x,y) \in \cE$ almost surely.	
\end{enumerate} In particular, if $\pi=(x_0,x_1,\dots,x_n)$ is any closed path and $\check{\pi}:=(x_n,x_{n-1},\dots,x_0)$ denotes its time reversal then
\[
\sum_{i=0}^{n-1} \omega(x_i,\Delta x_i) = \sum_{i=0}^{n-1} \check{\omega}(\check{x}_i,\Delta \check{x}_i),
\] where, for each $i=0,\dots,n-e$, we set $\check{x}_i:=x_{n-1}$ and $\Delta \check{x}_i:=(\check{x}_i,\check{x}_{i+1})$.
 \end{theorem}

In Section \ref{section-five}, we will see how we can use Theorem \ref{tone} to
compute certain passage times.

\section{The Bernoulli--Exponential distribution as the limit of Dirichlet random vectors}
\label{section-three}

To establish Theorem~\ref{tone}, an important step will be to show that the Bernoulli--Exponential distribution can be obtained as the weak limit of (suitably rescaled) Dirichlet random vectors. For convenience of the reader, we recall the definition of the Dirichlet distribution next.

\begin{definition} Given $M \in \N$ and a vector $\mathbf{a}=(a_1,\dots,a_M)$ with real positive entries, we define the \textit{Dirichlet distribution} of order $M$ with parameters $\mathbf{a}$, to be denoted by $\mathcal{D}_M(\mathbf{a})$, as the distribution of the random vector 
$U=(U_1,\dots,U_M)$ given by
\begin{equation}\label{eq:defU}
U_i:=\frac{W_i}{\sum_{j=1}^M W_j}, \quad i=1,\dots,M,
\end{equation} where $W_1,\dots,W_M$ are independent random variables with $W_i \sim \Gamma(a_i,1)$ for $i=1,\dots,M$, i.e.,
$W_i$ has density function
\[
f_{W_i}(w)= \frac{1}{\Gamma(a_i)}w^{a_i-1}\mathrm{e}^{-w}\mathrm{1}_{(0,\infty)}(w), 
\] where $\Gamma$ denotes the Gamma function.   
\end{definition}

\begin{remark}\label{rem:m=2} Whenever $M = 2$, we have that $U_i \sim \textrm{Beta}(a_i,a_1+a_2)$ and $U_2=1-U_1$.	
\end{remark}

The main result of this section is the following proposition.

\begin{proposition}\label{prop:conv} Fix $M \in \N$ and a vector $\mathbf{a}=(a_1,\dots,a_M)$ with real positive entries, and for each $\varepsilon > 0$ let us consider a random vector $U^{(\varepsilon)}=(U_1^{(\varepsilon)},\dots,U_M^{(\varepsilon)}) \sim \mathcal{D}_M(\varepsilon \mathbf{a})$. Then, as $\varepsilon \to 0^+$, 
\begin{equation}\label{eq:convd}
(-\varepsilon \log U_1^{(\varepsilon)},\dots,-\varepsilon \log U_M^{(\varepsilon)}) \overset{d}{\longrightarrow} X \sim \textrm{Be-Exp}_M(\mathbf{a}).
\end{equation}
\end{proposition}

To prove Proposition~\ref{prop:conv}, we will need two auxiliary results. The first is a standard property of Dirichlet random vectors, contained in Proposition~\ref{prop:ind} below, whose proof is omitted.
\begin{proposition}\label{prop:ind} For any $M \in \N$ and $a_1,\dots,a_M > 0$, if $W_1,\dots,W_M$ are independent random variables with $W_i \sim \Gamma(a_i,1)$ for $i=1,\dots,M$,  then the random vector $U$ defined in~\eqref{eq:defU} is independent of $\sum_{j=1}^M W_j$.
\end{proposition}

The second auxiliary result gives two key properties of the Bernoulli--Exponential law.

\begin{lemma}\label{lemma:prelim} For any $M \in \N_{\geq 2}$ and vector $\mathbf{a}=(a_1,\dots,a_M)$ with real positive entries, if~$E_1,\dots,E_M$ are independent random variables with $E_i \sim \textrm{Exp}(a_i)$ for each $i=1,\dots,M$ and we set $E^*:=\min\{ E_1,\dots,E_M\}$, then 
\begin{equation}\label{eq:exp1}
(E^* , E_1-E^*,\dots,E_M- E^*) \sim \textrm{Exp}(\sum_{j=1}^M a_j) \times \textrm{Be-Exp}_M (\mathbf{a}).
\end{equation} In particular, if $\overline{E} \sim \textrm{Exp}(\sum_{j=1}^M a_j)$ and $Z=(Z_1,\dots,Z_M) \sim \textrm{Be-Exp}_M (\mathbf{a})$ are independent then 
\begin{equation}\label{eq:exp2}
(\overline{E}+Z_1,\dots,\overline{E}+Z_{M}) \overset{d}{=} (E_1,\dots,E_M).
\end{equation} On the other hand, if $B=(B_1,\dots,B_M) \sim \mathrm{Be}_M(\mathbf{p}^{(\mathbf{a})})$ with $\mathbf{p}^{(\mathbf{a})}=(p_1^{(\mathbf{a})},\dots,p_M^{(\mathbf{a})})$ as in \eqref{eq:defp} is independent of $E=(E_1,\dots,E_n)$ then, conditional on $B_M=0$, we have 
\begin{equation}\label{eq:exp3}
((1-B_1)E_1,\dots,(1-B_{M-1})E_{M-1}) \sim \textrm{Be-Exp}_{M-1}(a_1,\dots,a_{M-1}).
\end{equation}
\end{lemma}

\begin{proof} To check \eqref{eq:exp1}, by separating into the different cases $\{E^*=E_i\}$ for $i=1,\dots,M$, a straightforward computation using the joint density of $E$ shows that for any $s,t_1,\dots,t_M \geq 0$ we have
\[
\P ( E^* \geq s\,,\,E_1-E^* \geq t_1,\dots,E_M -E^* \geq t_M) =\mathrm{e}^{-(\sum_{j=1}^M a_j) s}   \sum_{i=1}^M p_i^{(\mathbf{a})} \prod_{j \neq i} \mathrm{e}^{-a_j t_j },
\] from where \eqref{eq:exp1} now immediately follows. Equation \eqref{eq:exp2} is an immediate consequence of \eqref{eq:exp1}, so it only remains to show \eqref{eq:exp3}. To this end we observe that, since the $E_i$'s are independent, we only need to check that, conditional on $B_M=0$, we have 
\begin{equation}\label{eq:exp4}
(B_1,\dots,B_{M-1}) \sim \textrm{Be}_{M-1}(a_1,\dots,a_{M-1}).
\end{equation} But this is immediate to verify: indeed, for each $i=1,\dots,M-1$ we have
\[
\P( (B_1,\dots,B_{M-1})=\mathrm{e}_i | B_M=0) = \frac{\P ( B = \mathrm{e}_i )}{1 - \P ( B = \mathrm{e}_M)} = \frac{ a_i}{{\sum_{j=1}^{M-1} a_j}}
\] which readily implies \eqref{eq:exp4}. This concludes the proof.
\end{proof}

We are now ready to prove Proposition~\ref{prop:conv}.

\begin{proof}[Proof of Proposition~\ref{prop:conv}] We proceed by induction on $M$. The case $M=1$ is trivial, so let us consider first the case $M=2$. By Remark~\ref{rem:m=2}, the statement for $M=2$ is exactly that of \cite[Lemma~5.1]{BC17}. Thus, given $M > 2$, let us suppose that \eqref{eq:convd} holds for any vector with Dirichlet distribution of order $M-1$ and show that then \eqref{eq:convd} must also hold for all vectors with Dirichlet distribution of order $M$. 

Let $\mathbf{a}=(a_1,\dots,a_{M})$ be a vector with real positive entries and, for each $\varepsilon > 0$, consider a random vector $U^{(\varepsilon)}=(U_1^{(\varepsilon)},\dots,U_{M}^{(\varepsilon)}) \sim \mathcal{D}_{M}(\varepsilon\mathbf{a})$. Without loss of generality, we can assume that $U_i^{(\varepsilon)}=\frac{W_i}{\sum_{j=1}^M W_j}$ for each $i$, where $W_1,\dots,W_M$ are independent random variables with $W_j \sim \Gamma(\varepsilon a_i,1)$ for each $j$. In this case, for each $i=1,\dots,M-1$ we may write $U_i^{(\varepsilon)}$ as 
\[
U^{(\varepsilon)}_i = (1-B^{(\varepsilon)}) \cdot V_i^{(\varepsilon)} 
\] where 
\[
B^{(\varepsilon)}:=\frac{W_{M}}{\sum_{j=1}^{M-1} W_j + W_{M}}=U_{M}^{(\varepsilon)}\qquad \text{ and }\qquad V_i^{(\varepsilon)}:=\frac{W_i}{\sum_{j=1}^{M-1} W_j}.
\] Observe that, by definition of Dirichlet distribution together with Remark~\ref{rem:m=2}, Proposition~\ref{prop:ind} and the independence of the variables $W_j$, we have that $B^{(\varepsilon)}$ and $V^{(\varepsilon)}$ are independent, with distributions $B^{(\varepsilon)} \sim \textrm{Beta}(\varepsilon a_{M},\varepsilon \sum_{j=1}^{M-1} a_j)$ and $V^{(\varepsilon)} \sim \mathcal{D}_M(\varepsilon a_1,\dots,\varepsilon a_{M-1})$. Therefore, by inductive hypothesis and the case $M=2$, we conclude that 
\[
(-\varepsilon \log U_1^{(\varepsilon)},\dots,-\varepsilon \log U_{M}^{(\varepsilon)}) \overset{d}{\longrightarrow} (Y_1+Z_1,Y_1+Z_2,\dots,Y_1+Z_{M-1},Y_2)=:\widetilde{X}
\] where $Y=(Y_1,Y_2)$ and $Z=(Z_1,\dots,Z_{M-1})$ are independent and respectively given by
\[
(-\varepsilon \log (1-B^{(\varepsilon)}), -\varepsilon \log B^{(\varepsilon)}) \overset{d}{\longrightarrow} Y \sim \textrm{Be-Exp}_2\Big(\sum_{j=1}^{M-1} a_j, a_{M}\Big)
\] and
\begin{equation}\label{eq:Z}
(-\varepsilon \log V_1^{(\varepsilon)},\dots,-\varepsilon \log V_{M}^{(\varepsilon)}) \overset{d}{\longrightarrow} Z \sim \textrm{Be-Exp}_{M-1}(a_1,\dots,a_{M-1}).
\end{equation} Thus, to conclude the proof we only need to show that $\widetilde{X} \sim \textrm{Be-Exp}_{M}(\mathbf{a})$.

To this end, we observe that, by conditioning on whether $\widetilde{X}_{M}=0$ (that is, $Y_2=0$) or not, if $p^{(\mathbf{a})}_M$ is as in \eqref{eq:defp}, then for any $\mathbf{x}=(x_1,\dots,x_M)$ with nonnegative entries we can write
\[
F_{\widetilde{X}}(\mathbf{x})= p^{(\mathbf{a})}_M F_{(\overline{E}+Z_1,\dots,\overline{E}+Z_{M-1})}(x_1,\dots,x_{M-1}) + (1- p^{(\mathbf{a})}_M) F_Z(x_1,\dots,x_{M-1})F_{E_M}(x_{M}),
\] where $Z$ is as in \eqref{eq:Z}, $\overline{E} \sim \textrm{Exp}(\sum_{j=1}^{M-1} a_j)$ and $E_M \sim \textrm{Exp}(a_{M})$ are independent of $Z$ and, given any random vector $H$, we write $F_H$ to denote its cumulative distribution function.

On the other hand, if we take $X \sim \textrm{Be-Exp}_M(\mathbf{a})$ to be the random vector from \eqref{eq:defX}, then, by conditioning on whether $X_M=0$ (that is, $B_M=1$) or not, we have
\[
F_{X}(\mathbf{x})= p^{(\mathbf{a})}_M F_{E|_{M-1}}(x_1,\dots,x_{M-1}) + (1-p^{(\mathbf{a})}_M) F_{X^{(M-1)}}(x_1,\dots,x_{M-1})F_{E_M}(x_{M}),
\] where $E|_{M-1}:=(E_1,\dots,E_{M-1})$ are the first $M-1$ coordinates of the vector $E$ from \eqref{eq:defX}, and $X^{(M-1)}$ is distributed as $(X_1,\dots,X_{M-1})$ conditioned on $B_M=0$.

The result now follows immediately from Lemma \ref{lemma:prelim}, which in particular tells us that $E|_{M-1} \overset{d}{=} (\overline{E}+Z_1,\dots,\overline{E}+Z_M)$ and $X^{(M-1)} \overset{d}{=} Z$.
\end{proof}

\section{Proof of Theorem~\ref{tone}}
\label{section-four}
We will now use Proposition~\ref{prop:conv} together with the property 
of statistical reversibility of random walks in Dirichlet  
environments to prove our Theorem~\ref{tone}. 
We begin by recalling the model of random walks in Dirichlet environments.

Given a strongly connected finite directed graph $G=(V,\cE)$, for each $x \in V$ let 
\[
\cP_x=\{ \mathbf{p}_x=(p(x,e))_{e \in \cE_x} : p(x,e) \geq 0 \text{ for all }e \in \cE_x\,,\,\sum_{e \in \cE_x} p(x,e) =1\}
\] be the space of all probability vectors on $\cE_x$ and consider the product space $\Omega:=\prod_{x \in V} \mathcal{P}_x$ endowed with the usual product topology. Any element $\omega \in \Omega$ will be called an \textit{environment}, i.e., each $\omega=(\omega(x))_{x \in V}$ is a finite family of probability vectors $\omega(x)=(\omega(x,e))_{e \in \cE_x} \in \cP_x$. Given any $z \in V$ and $\omega \in \Omega$, we define the \textit{random walk in the environment $\omega$} starting at $z$ as the Markov chain $(X_n)_{n \in \N_0}$ on $G$ whose law $P_{z,\omega}$ is given by
\[
P_{z,\omega}( X_0 = z) =1 \hspace{1cm} \text{ and }\hspace{1cm}P_{z,\omega}(X_{n+1}= y | X_n = x ) = \omega(x,(x,y)) 
\] for all $x,y \in V$ such that $(x,y) \in \cE$.  If we now let the environment $\omega$ be chosen at random according to some Borel probability measure $\P$ on $\Omega$, we obtain the measure $P_x$ on $\Omega \times V^{\N_0}$ given by
\[
P_x(A \times B) := \int_A P_{x,\omega}(B)\mathrm{d}\P, 
\] for all Borel sets $A \subseteq \Omega$ and $B \subseteq V^{\N_0}$. The process $(X_n)_{n \in \N_0}$ under the measure $P_z$ is called the \textit{random walk in random environment} (RWRE) with \textit{environmental law} $\P$ (starting at $z$). The measure $P_z$ is known as the \textit{annealed law} of the RWRE, whereas $P_{z,\omega}$ for a fixed $\omega \in \Omega$ is referred to as the \textit{quenched law}. Finally, if given a family of positive weights $\mathbf{a}=(a_e)_{e \in \cE}$ we consider the environmental law 
\[
\P_{\cD}^{(\mathbf{a})}:= \prod_{x \in V} \mathcal{D}_{M_x} (\mathbf{a}_x),
\] where, as before, we write $\mathbf{a}_x:=(a_e)_{e \in \cE_x}$, then the resulting RWRE is called a \textit{random~walk in a Dirichlet environment} (RWDE) with parameters $\mathbf{a}$. 

Now let us fix a family of positive weights $\mathbf{a}=(a_e)_{e \in \cE}$ on $\cE$ and, for each $\varepsilon>0$, consider a Dirichlet 
random environment $\omega_\varepsilon$ on $G$ with parameter $\varepsilon \mathbf{a}$, i.e.,  a family of independent random vectors $\omega_\varepsilon=(\omega_\varepsilon(x))_{x \in V}$ where, for each $x \in V$, the vector $\omega_\varepsilon(x)=(\omega_\varepsilon(x,e))_{e \in \cE_x}$ has Dirichlet distribution of order $M_x$ with parameters $\varepsilon \mathbf{a}_x:=(\varepsilon a_e)_{e \in \cE_x}$. Then, since $\varepsilon a_e> 0$ for every $e \in \cE$ and $G$ is strongly connected, for $\P^{(\varepsilon \mathbf{a})}$-almost every $\omega_\varepsilon$ the associated walk $(X_n)_{n \in \N_0}$ is an irreducible finite-state Markov chain under $P_{z,\omega}$ and, as such, admits a unique stationary distribution $\pi^{\omega_\varepsilon} = (\pi^{\omega_\varepsilon}(x))_{x \in V}$. Following 
\cite{ST17}, we can now define the time-reversed environment 
$\check\omega_\varepsilon$ in the dual graph $\check G$ by
\begin{equation}
  \label{eone}
\check\omega_\varepsilon (y, \check{e}) =\frac{\pi^{\omega_\varepsilon}(x)}{\pi^{\omega_\varepsilon}(y)}\omega_\varepsilon(x,e), \qquad e=(x,y) \in \cE.
\end{equation} Since $\textrm{div}(\mathbf{a}) = 0$, it follows from \cite[Lemma~3]{ST17} that $\check{\omega}_\varepsilon$ is a Dirichlet random environment on the dual graph $\check{G}$ with parameters $\varepsilon \check{\mathbf{a}}$, for $\check{\mathbf{a}}$ as defined in \eqref{eq:defrev}.

If for each $e=(x,y) \in \cE$ we define the quantities
\[
\phi_\varepsilon(x,e):=-\varepsilon\log\omega_\varepsilon(x,e) \qquad\text{ and }\qquad
\check \phi_\varepsilon(y,\check e):=-\varepsilon\log\check {\omega}_\varepsilon(y,\check e), 
\] then \eqref{eone} gives 
\begin{equation}
 \label{two}
\check\phi_\varepsilon(y,\check e)=\upsilon_\varepsilon(
e)+ \phi_\varepsilon(x,e), 
\end{equation}
where 
\[
\upsilon_\varepsilon(e):=-\varepsilon \log\pi^{\omega_\varepsilon}(x) 
+\varepsilon \log\pi^{\omega_\varepsilon}(y). 
\] Finally, consider the random vector
\[
\Upsilon_\varepsilon = ( \phi_\varepsilon(x,e), \check{\phi}_\varepsilon(y,\check{e}),\upsilon_\varepsilon(e))_{e=(x,y) \in \cE}.
\] 
It follows from Proposition \ref{prop:conv} and \eqref{two} that the family $(\Upsilon_\varepsilon)_{\varepsilon > 0}$ is tight as $\varepsilon \to 0$, as shown in the following lemma. 

\begin{lemma} \label{l:two}
For any sequence $(\varepsilon_n)_{n \in \N} \subseteq (0,\infty)$ such that $\lim_{n \to \infty} \varepsilon_n = 0$ we have that the sequence
  $(\Upsilon_{\varepsilon_n})_{n \in \N}$ is tight.
\end{lemma}

\begin{proof} Since the graph $G$ is finite, that the family $(\phi_{\varepsilon_n}(x,e))_{e=(x,y) \in \cE}$ is tight (in $n \in \N$) follows from Proposition \ref{prop:conv}. Since $\check{\omega}_{\varepsilon_n} \sim \P^{(\varepsilon_n \check{\mathbf{a}})}_{\cD}$, similarly we have that $(\check{\phi}_{\varepsilon_n}(y,\check{e}))_{e=(x,y) \in \cE}$ is tight. In view of the relation in \eqref{two}, the tightness of $(\phi_{\varepsilon_n}(x,e),\check{\phi}_{\varepsilon_n}(y,\check{e}))_{e=(x,y) \in \cE}$ implies that of $(\upsilon_{\varepsilon_n}(e))_{e \in \cE}$. From here the result now follows.
 \end{proof}

As a consequence of Lemma \ref{l:two} we obtain that there exists some subsequence $(\varepsilon_n)_{n \in \N}$ such that $\lim_{n \to \infty} \varepsilon_n = 0$ 
and a random vector $\Upsilon=(\Upsilon^{(1)}(x,e),\Upsilon^{(2)}(y,\check{e}),\Upsilon^{(3)}(e))_{e=(x,y) \in \cE}$ satisfying that as $n \to \infty$,
\begin{equation}
\label{eq:convlim}
\Upsilon_{\varepsilon_n}\overset{d}{\longrightarrow} \Upsilon. 
\end{equation} We claim that $\Upsilon$ is the desired coupling.
Indeed, by Proposition~\ref{prop:conv} we have that $\Upsilon^{(1)} \sim \P^{(\mathbf{a})}_{\text{BE}}$ and $\Upsilon^{(2)} \sim \P^{(\check{\mathbf{a}})}_{\text{BE}}$. On the other hand, \eqref{two} and \eqref{eq:convlim} together imply that
\[
\Upsilon^{(2)}(y,\check{e})=\Upsilon^{(3)}(e) + \Upsilon^{(1)}(x,e)
\] for all $e=(x,y) \in \cE$. Finally, if $\pi=(x_0,x_1,\dots,x_{n-1},x_n)$ is a closed path, then given $\varepsilon > 0$ since $x_n=x_0$ we have that
\[
\sum_{i=0}^{n-1} \check{\phi}_\varepsilon(x_{i+1},(x_{i+1},x_i)) = -\varepsilon \log \left( \prod_{i=0}^{n-1} \frac{\pi^{\omega_\varepsilon}(x_i)}{\pi^{\omega_\varepsilon}(x_{i+1})}\omega(x,(x_i,x_{i+1}))\right) = \sum_{i=0}^{n-1} \phi_\varepsilon(x_{i},(x_{i+1},x_i)), 
\] so that by \eqref{two} the vector $(\upsilon_\varepsilon(e))_{e \in \cE}$ satisfies the closed loop property. It then follows from \eqref{eq:convlim} that $(\Upsilon^{(3)}(e))_{e \in \cE}$ must satisfy the closed loop property as well. Therefore, we see that $\Upsilon$ satisfies all the required properties and thus constitutes the desired coupling. This concludes the proof of Theorem~\ref{tone}.

\section{First passage return time}
\label{section-five}
We conclude by giving an example of how Theorem \ref{tone} can be used 
to do explicit computations of non-trivial first passage times. 

Consider the finite graph $G$ of Figure 1. The vertices are
classified into seven layers, plus an extra vertex which we call $y$.
The super-index $i$ in a vertex $x_k^i$ indicates to which layer it belongs, e.g. $x^i_k$ is the $k$-th vertex of the $i$-th layer. Let us fix three parameters $a_1,a_2,a_3>0$.  
We give the weight $a_1+a_2+a_3$ to the edge from $x^1_1$ to any vertex in
layer $2$. From layer $2$ to $3$ we give, for each $1\le k,j\le 3$, the weight $a_j$ to the edge connecting
$x^2_k$ to $x^3_j$. In general, given $j \in \{1,2,3\}$, for $2\le i \le 6$ we give the weight $a_j$ to any outgoing edge from a vertex $x^i_k$ in layer $i$ to the $j$-th vertex $x^{i+1}_j$ in layer $i+1$. Every outgoing edge from layer
$7$
to  $x^1_1$ or $y$ is assigned the weight $(a_1+a_2+a_3)/2$. The edge from $y$ to
$x^1_1$ has weight $a_1+a_2+a_3$, while the edge from $x_1^1$ to $y$
has weight $(a_1+a_2+a_3)/2$. The weights 
indicated in the edges of graph $G$ define a generalized 
Bernoulli--Exponential 
first passage percolation model with those weights.
Notice that the divergence condition $\text{div}({\bf a})=0$ is satisfied, where $\mathbf{a}$ denotes the vector of all such edge weights.

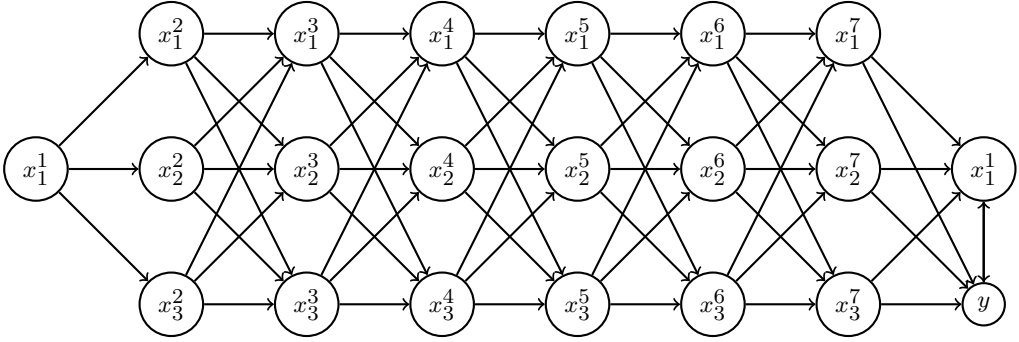
\begin{figure}
\begin{tikzpicture}[node distance={18mm}, thick, main/.style = {draw, circle}] 
\node[main] (1) {$x^1_1$}; 
\node[main] (2) [right of=1] {$x^2_2$}; 
\node[main] (3) [above of=2] {$x^2_1$}; 
\node[main] (4) [below  of=2] {$x^2_3$}; 
\node[main] (5) [right of=2] {$x^3_2$};
\node[main] (6) [right of=3] {$x^3_1$}; 
\node[main] (7) [right of=4] {$x^3_3$}; 
\node[main] (8) [right of=5] {$x^4_2$};
\node[main] (9) [right of=6] {$x^4_1$};
\node[main] (10) [right of=7] {$x^4_3$};
\node[main] (11) [right of=8] {$x^5_2$};
\node[main] (12) [right of=9] {$x^5_1$}; 
\node[main] (13) [right of=10] {$x^5_3$}; 
\node[main] (14) [right of=11] {$x^6_2$};
\node[main] (15) [right of=12] {$x^6_1$}; 
\node[main] (16) [right of=13] {$x^6_3$}; 
\node[main] (17) [right of=14] {$x^7_2$}; 
\node[main] (18) [right of=15] {$x^7_1$}; 
\node[main] (19) [right of=16] {$x^7_3$};
\node[main] (20) [right of=17] {$x^1_1$};
\node[main] (21) [right of=19] {$y$}; 
\draw[->] (1) --  (2); 
\draw[->] (1) -- (3); 
\draw[->] (1) -- (4);
\draw[->] (2) -- (5); 
\draw[->] (2) -- (6); 
\draw[->] (2) -- (7);
\draw[->] (3) -- (5); 
\draw[->] (3) -- (6); 
\draw[->] (3) -- (7);
\draw[->] (4) -- (5); 
\draw[->] (4) -- (6); 
\draw[->] (4) -- (7); 
\draw[->] (5) -- (8); 
\draw[->] (5) -- (9); 
\draw[->] (5) -- (10); 
\draw[->] (6) -- (8); 
\draw[->] (6) -- (9); 
\draw[->] (6) -- (10); 
\draw[->] (7) -- (8); 
\draw[->] (7) -- (9); 
\draw[->] (7) -- (10); 
\draw[->] (8) -- (11); 
\draw[->] (8) -- (12); 
\draw[->] (8) -- (13); 
\draw[->] (9) -- (11); 
\draw[->] (9) -- (12); 
\draw[->] (9) -- (13); 
\draw[->] (10) -- (11); 
\draw[->] (10) -- (12); 
\draw[->] (10) -- (13); 
\draw[->] (11) -- (14); 
\draw[->] (11) -- (15); 
\draw[->] (11) -- (16); 
\draw[->] (12) -- (14); 
\draw[->] (12) -- (15); 
\draw[->] (12) -- (16); 
\draw[->] (13) -- (14); 
\draw[->] (13) -- (15); 
\draw[->] (13) -- (16); 
\draw[->] (14) -- (17); 
\draw[->] (14) -- (18); 
\draw[->] (14) -- (19); 
\draw[->] (15) -- (17); 
\draw[->] (15) -- (18); 
\draw[->] (15) -- (19); 
\draw[->] (16) -- (17); 
\draw[->] (16) -- (18); 
\draw[->] (16) -- (19);
\draw[->] (17) -- (20); 
\draw[->] (18) -- (20); 
\draw[->] (19) -- (20); 
\draw[->] (17) -- (21); 
\draw[->] (18) -- (21); 
\draw[->] (19) -- (21); 
\draw[->] (20) -- (21); 
\draw[->] (21) -- (20); 
\end{tikzpicture}
\caption{Edges from layer $i$ to $j$ with $1\le i,j\le 6$ are assigned
the weights $a_1, a_2$ or $a_3$.}
\end{figure}


Define now the passage time,
\[
T_c(x_1^1, x_1^1):=\inf_{\pi\in\Pi_c(x^1_1,x^1_1)}
\left\{\sum_{i=0}^{n-1} w(x_i,\Delta x_i)\right\},
\]
where $\Pi_c$ is the set of paths which exit $x^1_1 $
through the right in Figure 1 (first jump is through any of the edges 
from $x^1_1$ to layer $2$) and finish entering again $x^1_1$ through 
layer $7$ or vertex $y$. By the statistical reversibility
of Theorem \ref{tone}, we have that

\begin{equation}
  \label{tthree}
T_c(x_1^1, x_1^1)=\inf_{\pi\in \check{\Pi}_c(x^1_1,x^1_1)}\left\{\sum_{i=0}^{n-1} \check w(x_i,\Delta x_i)\right\},
\end{equation}
where $\check{\Pi}_c$ is the set of paths in the dual graph $\check{G}$ which exit $x^1_1 $
with a first jump  through any of the edges 
from $x^1_1$ to layer $7$ or vertex $y$, and finish entering again $x^1_1$ through 
layer~$2$. But now note that for the reversed process there is always
a path of $0$ weight from any vertex in layer~$7$ to the vertex
$x^1_1$. Hence the minimum in the right-hand side of (\ref{tthree}) is
just the~minimum of the passage time in the dual graph $\check{G}$ from $x^1_1$ to layer $7$ (without revisiting $x^1_1$), which can be immediately computed and gives
\[
T_c(x^1_1,x^1_1)=\min\{\check w_{x^1_1,x^7_1}, \check w_{x^1_1,x^7_2}, \check
w_{x^1_1,x^7_3}
, \check w_{x^1_1,y}+\check w_{y,x^7_1},
\check w_{x^1_1,y}+\check w_{y,x^7_2},
\check w_{x^1_1,y}+\check w_{y,x^7_3}
\}.
\]

\subsection*{Acknowledgements}
  Alejandro Ram\'\i rez was supported by Fondecyt 1220396. Santiago Saglietti was supported by Fondecyt 11200690.

\end{document}